\documentclass[12pt]{amsart}
\newtheorem{theorem}{Theorem}[section]
\newtheorem{lemma}[theorem]{Lemma}
\newtheorem{corollary}[theorem]{Corollary}
\newtheorem{proposition}[theorem]{Proposition}
\newtheorem{remark}[theorem]{Remark}
\newtheorem{definition}[theorem]{Definition}

\newcommand{\ncom}{\newcommand}

\ncom{\rar}{\rightarrow}
\ncom{\lrar}{\longrightarrow}
\ncom{\ov}{\overline}
\ncom{\m}{\mbox}
\ncom{\sta}{\stackrel}
\ncom{\comx}{{\mathbb C}}
\ncom{\Z}{{\mathbb Z}}
\ncom{\Q}{{\mathbb Q}}
\ncom{\R}{{\mathbb R}}
\ncom{\G}{{\mathbb G}}
\ncom{\al}{\alpha}
\ncom{\p}{{\mathbb P}}
\ncom{\E}{{\mathbb E}}
\ncom{\N}{{\mathbb N}}
\ncom{\K}{{\mathbb K}}
\ncom{\Le}{{\mathbb L}}
\ncom{\A}{{\mathbb A}}
\ncom{\F}{{\mathbb F}}

\ncom{\f}{\frac}
\ncom{\cA}{{\mathcal A}}
\ncom{\cX}{{\mathcal X}}
\ncom{\cO}{{\mathcal O}}
\ncom{\cW}{{\mathcal W}}
\ncom{\cL}{{\mathcal L}}
\ncom{\cP}{{\mathcal P}}
\ncom{\cH}{{\mathcal H}}
\ncom{\cS}{{\mathcal S}}
\ncom{\cM}{{\mathcal M}}
\ncom{\cC}{{\mathcal C}}
\ncom{\cT}{{\mathcal T}}
\ncom{\cF}{{\mathcal F}}
\ncom{\cN}{{\mathcal N}}
\ncom{\cJ}{{\mathcal J}}
\ncom{\cV}{{\mathcal V}}
\ncom{\cZ}{{\mathcal Z}}
\ncom{\cU}{{\mathcal U}}
\ncom{\cSU}{{\mathcal S \mathcal U}}
\ncom{\cG}{{\mathcal G}}
\ncom{\cQ}{{\mathcal Q}}
\ncom{\cR}{{\mathcal R}}
\ncom{\eop}{{\hfill $\Box$}}

\begin{document}
\baselineskip=16pt

\title[]{Absolute Chow--K\"unneth decomposition for rational homogeneous bundles and for log homogeneous varieties}

\author[J. N. Iyer]{Jaya NN Iyer}

\address{The Institute of Mathematical Sciences, CIT
Campus, Taramani, Chennai 600113, India}
\address{Department of Mathematics and Statistics, University of Hyderabad, Gachibowli, Central University P O, Hyderabad-500046, India}
\email{jniyer@imsc.res.in}

\footnotetext{Mathematics Classification Number: 14C25, 14D05, 14D20, 14D21 }
\maketitle

\section{Introduction}


Suppose $X$ is a nonsingular projective variety of dimension $n$ defined over 
the complex numbers. Let $CH^i(X)\otimes \Q$ be the Chow group of codimension 
$i$ algebraic cycles modulo rational equivalence, with rational coefficients. 
Jacob Murre \cite{Mu2}, \cite{Mu3} has made the following conjecture which leads to a filtration on the rational Chow groups:

\textbf{Conjecture}: The motive $h(X):=(X,\Delta_X)$ of $X$ has a Chow-K\"unneth decomposition:
$$
\Delta_X= \sum_{i=0}^{2n}\pi_i \in CH^n(X\times X)\otimes \Q
$$
such that $\pi_i$ are orthogonal projectors (see \S \ref{CK}).

In this paper, \textit{absolute} Chow--K\"unneth decomposition (resply. projectors) is the same as Chow--K\"unneth decomposition (resply.projectors). We write 'absolute' to emphasize the difference with 'relative' Chow--K\"unneth projectors which will appear in the paper.

Some examples where this conjecture is verified are:  curves, surfaces,
a product of a curve and surface \cite{Mu1}, \cite{Mu3}, abelian varieties and abelian schemes \cite{Sh},\cite{De-Mu}, 
uniruled threefolds \cite{dA-Mul},  elliptic modular varieties
\cite{Go-Mu}, \cite{GHM2}), universal families over Picard modular surfaces \cite{MM} and finite group quotients (maybe singular) of 
abelian varieties \cite{Ak-Jo}, some varieties with a nef tangent bundles \cite{Iy}, open moduli spaces of smooth curves \cite{Iy-MS}, universal families over some Shimura surfaces \cite{Miller}.

In \cite{Iy}, we had looked at varieties which have a nef tangent bundle. Using the structure theorems of Campana and Peternell \cite{Ca-Pe} and Demailly-Peternell-Schneider \cite{DPS}, we know that such a variety $X$ admits a finite \'etale surjective cover $X'\rar X$ such that $X'\rar A$ is a bundle of smooth Fano varieties over an abelian variety. Furthermore, any fibre which is a smooth Fano variety necessarily has a nef tangent bundle. It is an open question \cite[p.170]{Ca-Pe} whether such a Fano variety is a rational homogeneous variety. They answered this question positively in dimension at most $3$. 
We showed in \cite{Iy} that whenever the \'etale cover is a relative cellular variety over $A$ or if it admits a relative Chow--K\"unneth decomposition, then $X'$ and $X$ have a Chow--K\"unneth decomposition. In particular, it holds for varieties with a nef tangent bundle of dimension at most $3$.

In this paper, we weaken the hypothesis on the cover $X'\rar A$ as above and obtain a Chow--K\"unneth decomposition
whenever $X'\rar A$ is a rational homogeneous bundle, over an abelian variety $A$. This strengthens the results in \cite{Iy} and if the open question \cite[p.170]{Ca-Pe} is answered positively in higher dimensions then we obtain a Chow--K\"unneth decomposition for all varieties which have a nef tangent bundle.
This question is answered positively in some higher dimensional cases also, see \cite[section 4]{Hwang} and references therein. Hence we obtain a Chow--K\"unneth decomposition for new cases as well in higher dimensions.
   
We state the result and proofs, in a  more general situation.

\begin{theorem}\label{homogvar}
Suppose $S$ is a smooth projective variety over the complex numbers. Let $G$ be a connected reductive algebraic group and let $Z$ be a rational $G$ homogeneous space over the variety $S$. Assume that
$S$ has a Chow-K\"unneth decomposition. Then the following hold:

a) the motive of $Z$ has an absolute Chow--K\"unneth decomposition.

b) the motive of the bundle $Z\rar S$ is expressed
as a sum of tensor products of summands of the motive of $S$ with
the twisted Tate motive.
\end{theorem}

One of the main observation in the proof is to note that a rational homogeneous bundle as above  is \'etale locally a relative cellular variety, using the fact that the formal deformations of a rational homogeneous variety are trivial (see Lemma \ref{Totaro}). Hence we can construct relative Chow--K\"unneth projectors (in the sense of \cite{De-Mu})  over  \'etale morphisms of $S$. These projectors
 lie in the subspace generated by the relative algebraic cells. The corresponding relative cohomology classes patch up since they lie in the subspace generated by the relative analytic cells. Hence the relative orthogonal projectors can be patched up as algebraic cycles to obtain relative projectors, in the rational Chow groups of the associated regular stack.
In this case, we show that the relative Chow--K\"unneth projectors over the regular stack descend to  relative Chow--K\"unneth projectors for $Z\rar S$ (see Corollary \ref{relCKstack}). The criterion of Gordon-Hanamura-Murre \cite{GHM2}, for obtaining absolute Chow--K\"unneth projectors from relative Chow--K\"unneth projectors can be directly applied, see Proposition \ref{kd}.

A similar proof also holds for a class of log homogeneous varieties studied by M. Brion \cite{Brion}. A log homogeneous variety consists of a pair
$(X,D)$, where $X$ is a smooth projective variety and $D$ is a normal crossing divisor on $X$, with the following property. The variety $X$ is said to be log homogeneous with respect to $D$ if the associated logarithmic tangent bundle $\cT_X(-D)$ is generated by its global sections. It follows that $X$ is almost homogeneous under the connected automorphism group $G:= Aut^0(X,D)$, with boundary $D$. With notations as above, we show 

\begin{theorem}\label{loghomogvar}
Suppose $X$ is log homogeneous with respect to a normal crossing divisor $D$. Then $X$ has a Chow--K\"unneth decomposition. Moreover, the motive of $X$ is expressed as a sum of tensor products of the summands of the motive of its Albanese reduction, with the twisted Tate motive.
\end{theorem}
See Theorem \ref{logCK}.

The proof uses the classification of log homogeneous varieties by Brion \cite{Brion}. The fibres of the Albanese morphism are smooth spherical varieties. In this case we check that \'etale local triviality of the Albanese fibration holds. The proof of Theorem \ref{loghomogvar} relies on the algebraicity of the cohomology of the spherical varieties, similar to Theorem \ref{homogvar}, and applying the criterion of \cite{GHM2}.

{
\Small
Acknowledgements:  We thank B. Totaro for pointing out some errors in the previous version and for  helpful suggestions. Thanks are also due to J-M.Hwang for informing us about the status of Campana-Peternell conjecture in higher dimensions.
}
\section{Preliminaries}
We work over the field of complex numbers in this paper. We begin by recalling the standard constructions of the category of motives.
Since this is fairly discussed in the literature, we give a brief account and refer to \cite{Mu2}, \cite{Sc} for details. 

\subsection{Category of motives}
The category of nonsingular projective varieties over $\comx$ will be denoted by $\cV$.
For an object $X$ of $\cV$, let $CH^i(X)_\Q=CH^i(X)\otimes \Q$ denote the rational Chow group of codimension $i$ algebraic cycles modulo rational equivalence.
Suppose $X,Y\in Ob(\cV)$ and $X=\cup X_i$ be a decomposition into connected components $X_i$ and $d_i=\m{dim }X_i$.
Then $\m{Corr}^r(X,Y)= \oplus_i CH^{d_i+r}(X_i\times Y)_\Q$ is the group of correspondences of degree $r$ from $X$ to
$Y$. 
\vskip .3cm 
 We will use the standard framework of the category of Chow motives $\cM_{rat}$ in this paper and refer to \cite{Mu2} for details.
We denote the category of motives $\cM_{\sim}$, where $\sim$ is any equivalence, for instance $\sim$ is homological or numerical equivalence.
When $S$ is a smooth variety, we also consider the category of relative 
Chow motives $CH\cM(S)$ which was introduced in \cite{De-Mu} and \cite{GHM}. When $S=\m{ Spec } \comx$ then the category $CH\cM(S)= \cM_{rat}$. 

\subsection{Chow--K\"unneth decomposition for a variety}\label{CK}

Suppose $X$ is a nonsingular projective variety over $\comx$ of dimension $n$.
Let $\Delta_X\subset X\times X$ be the diagonal.
Consider the K\"unneth decomposition of $\Delta$ in the Betti Cohomology:
$$
\Delta_X = \oplus_{i=0}^{2n}\pi_i^{hom}
$$
where $\pi_i^{hom}\in H^{2n-i}(X)\otimes H^i(X)$.

\begin{definition}
The motive of $X$ is said to have K\"unneth decomposition if each of the
classes $\pi_i^{hom}$ are algebraic and are projectors, i.e., $\pi_i^{hom}$ is the image of an algebraic cycle $\pi_i$
under the cycle class map from the rational Chow groups to the Betti Cohomology and satisfying $\pi_i\circ \pi_i=\pi_i$ and  $\Delta_X = \oplus_{i=0}^{2n}\pi_i$ in the rational Chow ring of $X\times X$.  
The algebraic projectors $\pi_i$ are called as the algebraic K\"unneth projectors.
\end{definition}

\begin{definition}
The motive of $X$ is furthermore said to have a Chow--K\"unneth decomposition if the algebraic K\"unneth projectors are orthogonal projectors, i.e.,
$\pi_i\circ \pi_j=\delta_{i,j}\pi_i$ and $\Delta_X = \oplus_{i=0}^{2n}\pi_i$ in the rational Chow ring of $X\times X$.
\end{definition}

\section{Rational homogeneous bundles over a variety}

In this section, we firstly recall the motive of a rational homogeneous variety and later construct relative Chow--K\"unneth projectors for a bundle of homogeneous varieties. The criterion of \cite{GHM2} can then be applied to obtain absolute Chow--K\"unneth projectors on the total space of the bundle. For this purpose, we need to show that the bundle is \'etale locally trivial and check patching conditions over the etale coverings. We begin by recalling the motive of a rational homogeneous variety.

\subsection{The motive of a rational homogeneous space}

Suppose $F$ is a rational homogeneous variety. Then $F$ is identified as a quotient $G/P$, for some reductive linear algebraic group $G$ and $P$ is a
parabolic subgroup of $G$. 
Notice that $F$ is a cellular variety, i.e., it has a cellular decomposition
$$
\emptyset=F_{-1}\subset F_0 \subset...\subset F_n=F
$$
such that each $F_i\subset F$ is a closed subvariety and $F_i-F_{i-1}$ is an 
affine space.

Then we have
\begin{lemma}\cite[Theorem, p.363]{Ko}
\label{le.-motF}
The Chow motive $h(F)=(F,\Delta_F)$ of $F$ decomposes as a direct sum of twisted Tate motives
$$
h(F)=\bigoplus_{\omega}\Le^{\otimes \m{dim }\omega}.
$$
Here $\omega$ runs over the set of cells of $F$.
\end{lemma}

In particular, this says that the Chow--K\"unneth decomposition holds for $F$. Next we consider bundles of homogeneous spaces $Z\rar S$ over a smooth variety $S$. We want to describe the Chow motive of $Z$ in terms of the Chow motive of $S$, up to some Tate twists. 
For this , we need to show \'etale local triviality of $Z\rar S$, and we discuss it in the next subsection.

\subsection{The \'etale local triviality of a rational homogeneous bundle}\label{etaleloc}

Suppose $Z\rar S$ is a smooth projective morphism and the base variety $S$ is smooth and projective.

By \'etale local triviality, we mean that there exist \'etale morphisms $p_\alpha:U_\alpha\rar S$ such that the pullback bundle
$$
Z_{U_\alpha}:=Z\times_S{U_\alpha}\rar U_\alpha
$$
is a Zariski trivial fibration and the images of $p_\alpha$ cover $S$, i.e., $\cup_\alpha p_\alpha(U_\alpha)=S$. Here $\alpha$ runs over some indexing set $I$.
Consider a rational homogeneous bundle $f:Z\lrar S$,
i.e., $\pi$ is a smooth projective morphism and any fibre $\pi^{-1}y$ is a rational homogeneous variety $G/P$. Here $G$ is a reductive linear algebraic group  and $P\subset G$
is a parabolic subgroup. Assume that $S$ is a smooth complex projective variety.

In the following discussion, we note that an \'etale cover \{$U_\alpha$\} as above exists for a rational homogeneous bundle $Z\rar S$.

\begin{lemma}\label{Totaro} There are \'etale open  sets $p_\alpha:U_\alpha \rar S$ (satisfying $\cup_\alpha p_\alpha(U_\alpha)=S$), such that
the pullback bundle $Z_{U_\alpha}\rar U_\alpha$ is a Zariski trivial fibration.
\end{lemma}
\begin{proof}
We need to note that the formal deformations of a rational homogeneous variety are trivial.
This is just a consequence of the well-known Bott's vanishing theorem:
$H^1(G/P, T) =0$. The assertion on \'etale local triviality follows from \cite[Proposition 2.6.10]{Sernesi}.

\end{proof}

Our aim is to obtain relative Chow--K\"unneth projectors for the bundle $Z/S$. For this purpose, we first construct relative projectors over the \'etale coverings of $Z\rar S$ and check the patching conditions. This requires us to use the language of stacks which enables us to descend the projectors down to $Z\rar S$. Hence in the following subsection, we recall some facts on regular stacks and the relationship of the rational Chow groups/cohomology of stacks with that of its coarse moduli space.
These facts will be essentially applied to the simplest situation---the rational homogeneous bundle $Z\rar S$. Also, the patching will be used for
\'etale open sets of $Z$ which are of the type $Z_{U_\alpha}:=Z\times_S U_\alpha$, for \'etale morphisms $U_\alpha\rar S$.
In this context, it is possible to avoid stacks, since the regular stack  associated to the \'etale coverings is again $Z$.
But we use the stacks, essentially to say that the algebraic cells which live in the fibres of $Z\rar S$ patch together over the \'etale coverings. This will be needed in the proof of Lemma \ref{pushforwardiso}.

We remark that more general patching statements might also hold for other varieties, using stacks. However we do not know concrete examples as yet, where it can be checked.

\subsection{Chow groups of an \'etale site}

Mumford, Gillet (\cite{Mumford},\cite{Gillet}) have defined Chow groups for Deligne--Mumford stacks and more generally for any algebraic stack $\cX$. Furthermore, intersection products are defined whenever $\cX$ is a regular stack . 
Let $\cX$ be a regular stack. The coarse moduli space of $\cX$ is denoted by $X$ and $p:\cX\rar X$ be the projection.
So from \cite[Theorem 6.8]{Gillet}, the pullback $p^*$ and pushforward maps $p_*$ establish a 
ring isomorphism of rational Chow groups

\begin{equation}\label{ringiso}
CH^{\ast}(\cX)_{\Q }\cong CH^{\ast}(X)_{\Q }.
\end{equation}

This can be applied to the product $p\times p:\cX\times \cX\rar X\times X$, to get a ring isomorphism
\begin{equation}\label{ringiso2}
CH^{\ast}(\cX\times \cX)_{\Q }\cong CH^{\ast}(X\times X)_{\Q }.
\end{equation}
Assume that $X$ is a smooth projective variety.
Then these isomorphisms also hold in the rational singular cohomology of $\cX$ and $\cX\times \cX$ (for example, see \cite{Behrend}):
\begin{equation}\label{behrend1}
H^{\ast}(\cX,\Q )\cong H^{\ast}(X,\Q ).
\end{equation}
and
\begin{equation}\label{behrend2}
H^{\ast}(\cX\times \cX,\Q )\cong H^{\ast}(X\times X,\Q ).
\end{equation}

Via these isomorphisms, we can pullback the K\"unneth decomposition of the diagonal class in $H^{2n}(X\times X,\Q)$ to a decomposition 
of the diagonal class of $\cX$ in $H^{2n}(\cX\times \cX,\Q)$, and whose components we refer to as the K\"unneth components of $\cX$.

Given a smooth variety $X$, consider an atlas $\sqcup_{\alpha \in I} U_\alpha$ of $X$ such that $p_\alpha:U_\alpha\rar X$ is an \'etale morphism, for each $\alpha \in I$, and the images of $p_\alpha$ cover $X$. Then one can associate a $Q$-variety \cite{Mumford} to this atlas. Furthermore,  by \cite[Proposition 9.2]{Gillet}, there is a regular stack
$\cX$ associated to this data such that $X$ is its coarse moduli space, i.e., there is a projection
$$
p:\cX\rar X.
$$ 
In this case, we note that the regular stack $\cX$ is the same as the variety $X$. 
Hence the isomorphisms in \eqref{ringiso}, \eqref{ringiso2}, \eqref{behrend1} and \eqref{behrend2} trivially hold for the projection $p$. 
More precisely, we have
$$
CH^*(\cX)_\Q= CH^*(X)
$$
and
$$
H^*(\cX,\Q)=CH^*(X,\Q).
$$

\subsection{The motive of a rational homogeneous bundle}

Suppose $Z\rar S$ is a rational homogeneous bundle over a smooth projective variety $S$. Let $S^{et}$ be the \'etale site on $S$, together with the natural morphism of the sites $f:S^{et}\rar S$. Here $S$ is considered with the Zariski site. Consider the pullback bundle
$$
Z^{et}:=Z\times_S S^{et}\rar S^{et}
$$
over $S^{et}$.

Since we are dealing with a rational homogeneous bundle, we can describe these covers explicitly as follows;
by Lemma \ref{Totaro}, the pullback bundles $Z_{U_\alpha}\rar U_\alpha$, for $\alpha\in I$, are Zariski  trivial. In other words,  
$Z_{U_\alpha}=F\times U_\alpha$, where $F$ is a typical fiber of $Z\rar S$.
Hence $Z_{U_\alpha} \rar U_\alpha $ is a relative cellular variety for each $\alpha\in I$. 

The description of the rational Chow groups
of relative cellular spaces $\pi:X\rar T$ is given by B. Koeck \cite{Ko} (see also \cite[Theorem 5.9]{Ne-Za}), which is stated for the higher Chow groups:

Suppose $X\rar T$ is a relative cellular space.

Then there is a sequence of closed embeddings
\begin{equation}\label{cellulardecomposition}
\emptyset=Z_{-1}\subset Z_{0} \subset...\subset Z_n=X
\end{equation}
such that $\pi_k:Z_k\lrar T$ is a flat projective $T$-scheme. Furthermore, for any $k=0,1,...,n$, the open
complement $Z_k-Z_{k-1}$ is $T$-isomorphic to an affine space $\A^{m_k}_{T}$ of relative dimension $m_k$.
Denote $i_k:Z_k\hookrightarrow X$.

\begin{theorem}\label{Koeck}
For any $a,b\in \Z$, the map
\begin{eqnarray*}
\bigoplus_{k=0}^nH_{a-2m_k}(T,b-m_k) & \lrar & H_a(X,b) \\
(\al_0,...,\al_n) & \mapsto & \sum_{k=0}^n(i_k)_*\pi_k^*\al_k
\end{eqnarray*}
is an isomorphism.
Here $H_a(T,b)=CH_b(T,a-2b)$ are the higher Chow groups of $T$.
\end{theorem} 
\begin{proof}
See \cite[Theorem, p.371]{Ko}.
\end{proof}
The above theorem can equivalently be restated to express the rational Chow groups of $X$ as 
\begin{equation}\label{koeck}
CH^r(X)_\Q= \bigoplus_{k=0}^r (\oplus_\gamma\Q[\omega_k^\gamma]).f^*CH^k(T)_\Q,
\end{equation}
Here $\omega_k^\gamma$ are the $r-k$ codimensional relative cells and $\gamma$ runs over the indexing set of $r-k$ codimensional relative cells in the $T$-scheme $X$.

We now apply this theorem to our situation: we have a homogeneous bundle $Z\rar S$ and an \'etale atlas $S^{et}:=\sqcup_\alpha U_\alpha \rar S$, such that $Z_{U_\alpha}\rar U_\alpha$ is trivial.

\begin{lemma}\label{relativecell}
Given a Zariski  trivial homogeneous bundle $p_\alpha:Z_{U_\alpha}\rar U_\alpha$, the rational Chow groups are described as follows:
$$
CH^r(Z_{U_\alpha})_\Q = \bigoplus_{k=0}^{r} (\oplus_\gamma\Q[\omega_k^\gamma]).p^*_\alpha CH^k(U_\alpha)_\Q.
$$
\end{lemma}
\begin{proof} 
Since the homogeneous bundle $p_\alpha:Z_{U_\alpha}\rar U_\alpha$ is a Zariski trivial bundle, it is a relative cellular variety. Hence the above Theorem \ref{Koeck} can be applied and it gives a natural isomorphism
$$
CH^r(Z_{U_\alpha})_\Q = \bigoplus_{k=0}^{r} (\oplus_\gamma\Q[\omega_k^\gamma]).f^*_\alpha CH^k(U_\alpha)_\Q.
$$
Equivalently, since $Z_{U_\alpha}=F\times U_\alpha$, we have
the equality (see \cite[Theorem 2]{Fulton}):
\begin{equation}\label{relchow}
CH^r(Z_{U_\alpha})_\Q=CH^r(F\times U_\alpha)_\Q=\sum_{p,q,p+q=r}CH^p(F)_\Q.CH^q(U_\alpha)_\Q.
\end{equation}
Here $F$ is a typical fibre of $Z\rar S$ which is cellular variety.
This gives the assertion.
\end{proof}
For our applications, it suffices to consider the piece $k=0$, which consists of only the relative algebraic cells of codimension $r$, namely,
$$
RCH^r(Z_{U_\alpha})_\Q := \oplus_\gamma\Q[\omega_0^\gamma].
$$
In other words, we only look at the subgroup consisting of the direct summand 
$$
CH^r(F)\subset CH^r(F\times U_\alpha),
$$
in \eqref{relchow}.

A similar equality as in \eqref{relchow}, holds in the rational singular cohomology of $Z_{U_{\alpha}}\rar U_\alpha$. 
So we can also define the piece
$$
RH^{2r}(Z_{U_\alpha})_\Q := \oplus_\gamma\Q[\omega_0^\gamma]
$$
in the rational singular cohomology of $Z_{U_\alpha}$ and the piece
$$
RH^{2r}(Z)_\Q := \oplus_\gamma\Q[\omega_0^\gamma]
$$
as a subspace of the rational Betti cohomology $H^{2r}(Z,\Q)$, generated by the relative analytic cells $\omega_0^\gamma$. Here, we use the fact that $Z\rar S$ is locally trivial in the analytic topology and there is a analytic cellular decomposition similar to \eqref{cellulardecomposition}.

\begin{lemma}\label{pushforwardiso}
The cycles $\omega_0^\gamma$ in $RCH^*(Z_{U_\alpha})_\Q$ patch together in the \'etale site to determine a subspace $RCH^*(Z)_\Q$ of
$CH^*(Z)_\Q$, generated by the patched cycles and which maps isomorphically onto the subspace $RH^{2r}(Z)_\Q \subset H^{2r}(Z,\Q)$, under the cycle class map  
$$
CH^*(Z)_\Q \rar  H^{2*}(Z,\Q).
$$
\end{lemma}
\begin{proof}
Note that the cycles $\omega_0^\gamma \in  RCH^*(Z_{U_\alpha})_\Q$ patch together as analytic cycles in the \'etale site and determine a subspace
$RH^{2r}(Z)_\Q\subset H^{2r}(Z,\Q)$, 
 
Since the fiber $F$ is a cellular variety, there is a natural isomorphism
\begin{equation}\label{reliso}
RCH^*(Z_{U_\alpha})_\Q\sta{\simeq}{\rar} RH^{2*}(Z_{U_\alpha})_\Q
\end{equation}
between the $0$-th piece of the rational Chow group and the relative Betti cohomology, for each $\alpha$. 

Via the isomorphism in \eqref{reliso}, the patching conditions required over the \'etale site, to define the piece $RCH^{2r}(Z)_\Q$ are the same as those for $RH^{2r}(Z)_\Q$. More precisely, the patching conditions are given in  \cite[\S 4]{Gillet}. The identification in \eqref{reliso} together with the fact that the patching conditions are fulfilled for the singular cohomology of the  \'etale site, says that the cycles $\omega_0^\gamma$ patch together to give
a class in $RH^{2r}(Z)_\Q$, and hence they also patch together to give a class in $RCH^{2r}(Z)_\Q$. These patched classes
generate the $\Q$-subspace  $RCH^{2r}(Z)_\Q\subset  CH^*(Z)_\Q$ and which maps isomorphically onto the subspace $RH^{2r}(Z)_\Q \subset H^{2r}(Z,\Q)$ under the cycle class map.
\end{proof}

\begin{corollary}\label{pushforwardiso2}
There is a canonical isomorphism
$$
RCH^r(Z)_\Q \simeq RH^{2r}(Z)_\Q.
$$
between the rational Chow groups and the rational cohomology generated by the relative cells. 
\end{corollary}
\eop

Let $n:=\m{dim}(Z/S)$.

\begin{corollary}\label{relCKstack}
The bundle $Z\rar S$ has a relative Chow--K\"unneth decomposition, in the sense of \cite{GHM}.
\end{corollary}
\begin{proof}
This is an application of Lemma \ref{pushforwardiso}, applied to the relative product $Z\times_SZ\rar S$. We  
notice that the relative orthogonal K\"unneth projectors in $H^{2n}(Z\times_S Z,\Q)$ lift to relative orthogonal projectors in $H^{2n}(Z_{U_\alpha}\times_{U_\alpha}  Z_{U_\alpha},\Q)$ and which add to the relative diagonal cycle.
Now we note that the relative diagonal $\Delta_{Z/S}$ and its orthogonal K\"unneth components 
actually lie in the piece $RH^{2n}(Z\times_S Z)_\Q$ (generated by the relative algebraic cells) and under the isomorphisms in \eqref{behrend1}, \eqref{behrend2}, lift to an orthogonal decomposition
$$
\Delta_{Z/S}=\sum_{i=0}^{2n}\Pi_i \,\in\, RH^{2n}(Z\times_{S} Z)_\Q 
$$
over the \'etale site, i.e, over $Z_{U_\alpha}\times_{U_\alpha} Z_{U_\alpha}$, for each $\alpha\in I$.
Now apply Corollary \ref{pushforwardiso2} to the product space $Z_{U_\alpha}\times_{U_\alpha}Z_{U_\alpha} \rar U_\alpha$, to lift the above orthogonal projectors to orthogonal algebraic projectors in $RCH^n(Z_{U_\alpha}\times_{U_\alpha}Z_{U_\alpha})_\Q$, and these patch to give relative Chow--K\"unneth projectors and a relative Chow--K\"unneth
decomposition
$$
\Delta_{Z/S}=\sum_{i=0}^{2n}\Pi_i \,\in\, CH^n(Z\times_S Z)_\Q.
$$ 
\end{proof}

\begin{proposition}\label{kd}
Suppose $Z\rar S$ is a rational homogeneous bundle over a smooth variety $S$.
Then the motive of the bundle $Z\rar S$ is expressed
as a sum of tensor products of summands of the motive of $S$ with
the twisted Tate motive. More precisely,
the motive of $Z$ can be written as
$$
h(Z)= \bigoplus_i h^i(Z)
$$
where $h^i(Z)= \bigoplus_{j+k}r_{\omega_\al}.\Le^j\otimes h^k(S)$.
Here $r_{\omega_\al}$ is the number of $j$-codimensional cells on a fibre $\F$ 

In particular, if $S$ has a Chow--K\"unneth decomposition then $Z$ also 
admits an absolute Chow--K\"unneth decomposition.
\end{proposition}
\begin{proof}
By Corollary \ref{relCKstack}, we know that the bundle $Z/S$ has a relative Chow--K\"unneth decomposition. Since the map $Z\rar S$ is a smooth morphism and the fibres of $Z\rar S$ have only algebraic cohomology, we can directly apply the criterion in \cite[Main theorem 1.3]{GHM2}, to get absolute Chow--K\"unneth projectors for $Z$ and the decomposition stated above (for example, see \cite[Lemma 3.2, Corollary 3.3]{Iy}). 
\end{proof}

\begin{remark}
Suppose $X$ is a smooth projective variety with a nef tangent bundle. Then by \cite{Ca-Pe},\cite{DPS}, we know that there is an \'etale cover $X'\rar X$ of $X$ such that $X'\rar A$ is a smooth morphism over an abelian variety $A$, whose fibres are smooth Fano varieties with a nef tangent bundle.
It is an open question \cite[p.170]{Ca-Pe}, whether such a Fano variety is a rational homogeneous variety. A positive answer to this question, together with Proposition \ref{kd}, will give absolute Chow--K\"unneth projectors for all varieties with a nef tangent bundle. See also \cite[section 4]{Hwang} for a discussion on new cases where this question is answered positively.
\end{remark}


\section{Chow--K\"unneth decomposition for log homogeneous varieties}

Log homogeneous varieties were introduced by M. Brion \cite{Brion}. Suppose $X$ is a smooth projective variety and $D\subset X$ is a normal crossing divisor.
Then $X$ is said to be log homogeneous with respect to $D$ if the logarithmic tangent bundle $\cT_X(-D)$ is generated by its global sections. Then $X$ is almost homogeneous under the connected automorphism group $G:=Aut^0(X,D)$, with boundary $D$. The $G$-orbits in $X$ are exactly the strata defined by $D$, in particular their number is finite.

A classification of log homogeneous varieties is given by Brion which says:

\begin{theorem}\label{classifylog}
Any log homogeneous variety $X$ can be written uniquely as $G\times^I Y$, where 

1) $G$ is  connected algebraic group, 

2) $I \subset G$ is a closed subgroup containing $G_{aff}$ as a subgroup of finite index,

3) choose any Levi subgroup $L\subset G_{aff}$, $Y$ is a complete smooth $I$-variety containing an open $L$-stable
subset $Y_L$ such that the $L$-variety $Y$ is spherical and the projection
\begin{equation}\label{albanesefib}
X\rar G/I =:A
\end{equation}
is the Albanese morphism.
\end{theorem} 
\begin{proof}
See \cite[Theorem 3.2.1]{Brion}.
\end{proof}

Recall that a smooth spherical variety $Y$ is a $G$-variety such that the Borel subgroup $B$ of $G$ has an open dense orbit in $Y$. It is known that $Y$ contains a finite number of $B$-orbits. Since we are looking at varieties defined over $\comx$, it follows that a spherical variety is a linear variety (in the sense of \cite[Addendum, p.5]{Totaro}).
In particular, we have

\begin{lemma}\label{fulton} 
Suppose $Y$ is a smooth complete spherical variety.
Then there is an isomorphism
$$
CH^i(Y) \sta{\simeq}{\rar} H^{2i}(Y,\Z)
$$
for each $i$.
\end{lemma}
\begin{proof}
See \cite[Corollary to Theorem 2]{Fulton}.
\end{proof}

\begin{lemma}\label{CKspherical}
Suppose $Y$ is a smooth complete spherical variety. Then $Y$ has a Chow--K\"unneth decomposition.
\end{lemma}
\begin{proof}
This follows from Lemma \ref{fulton} and the construction of orthogonal projectors given in \cite[Lemma 5.2]{Iy-MS}.
\end{proof}

We will show that $X$ has a Chow--K\"unneth decomposition under the following assumption;

\begin{theorem}\label{logCK}
Suppose $X$ is a log homogeneous variety.
Then the variety $X$ has a Chow--K\"unneth decomposition. Moreover, the motive of $X$ is expressed as a sum of tensor products of the summands of the motive of its Albanese reduction, with the twisted Tate motive.
\end{theorem}
\begin{proof}
With notations as in Theorem \ref{classifylog},
suppose the spherical variety $Y$ is a Fano variety. Then, by \cite[Proposition 4.2 i)]{Bien-Brion}, we have the vanishing $H^1(Y,T_{Y})=0$. In particular, this implies that the formal deformations of $Y$ are trivial. 
Hence, by \cite[Proposition 2.6.10]{Sernesi}, the Albanese fibration in \eqref{albanesefib} is \'etale locally trivial.
In general, consider the Albanese fibration
$$
X = G \times^I Y \rar G/I = A
$$
which is easily seen to be etale locally trivial. The following explanation is due to B. Totaro:
notice that all the fibers of this morphism are isomorphic
to Y. In more detail, this morphism
is etale locally trivial because the morphism $G \rar G/I$ is etale locally
trivial, which is a standard fact about the quotient of an algebraic
group by a smooth closed subgroup. See the discussion of homogeneous
spaces in \cite[6.14]{Borel}.

Hence we can apply the methods from the previous section.
By Lemma \ref{CKspherical}, relative Chow--K\"unneth projectors can be constructed for Zariski trivialisations of \eqref{albanesefib} over \'etale covers $U_\alpha\rar A$. Hence the proof of Proposition \ref{kd} applies to this situation. Indeed, Lemma \ref{relativecell} holds for a relative spherical variety over $U_\alpha$. This can be applied
to the Albanese fibration in \eqref{albanesefib} over \'etale morphisms where it is Zariski trivial. In this case, the following piece of the rational Chow ring $RCH^*(U_\alpha\times Y)_\Q$ is identified with the Chow ring $CH^*(Y)_\Q$. 
A formula similar to \eqref{koeck} holds for the Chow groups of $U_\alpha \times Y$, since $Y$ is cellular, see \cite[Theorem 2]{Fulton}.
Hence, by Lemma \ref{fulton},  $CH^*(U_\alpha\times Y)_\Q\simeq  H^{2*}(U_\alpha\times Y)_\Q$. Similarly Lemma \ref{pushforwardiso} and Corollary \ref{pushforwardiso2} hold for \eqref{albanesefib} over \'etale morphisms. The rest of the arguments are the same as given for a rational homogeneous bundle.
\end{proof}


\end{document}